\documentclass[12pt]{amsart}
\usepackage{amssymb}

 \newcommand{\Z}{\mathbb Z}\newcommand{\nn}{\mathbb N}

\DeclareMathOperator{\Per}{Per}

\theoremstyle{plain} \newtheorem{thm}{Theorem}
\newtheorem*{thm*}{Theorem}
\newtheorem{cor}[thm]{Corollary} \newtheorem{prop}[thm]{Proposition}
\newtheorem{lemma}[thm]{Lemma}

\theoremstyle{definition} \newtheorem{defn}[thm]{Definition}

\newtheorem{ex}[thm]{Example} 

\theoremstyle{remark} 

\def\category#1{\text{\sc #1}}
\newcommand{\Szym}{\category{Szym}}
\newcommand{\RelF}{\category{Rel}_\mathrm{f}}

\DeclareMathOperator{\End}{End}
\newcommand{\id}{\operatorname{id}}
\newcommand{\poX}{\overline X}
\newcommand{\porder}{R_\ge}
\newcommand{\Rprodxi}{\overline{R}_{\xi}}
\newcommand{\Rcanon}[1]{\tilde{R}_{#1}}
\newcommand{\Rcanonxi}{\Rcanon{\xi}}
\newcommand{\Xcanon}[1]{\tilde{X}_{#1}}
\newcommand{\Xcanonxi}{\Xcanon{\xi}}

\numberwithin{thm}{section}

\begin{document}

\title[Shift Equivalence for Boolean Matrices]{A Complete Invariant for Shift Equivalence for Boolean Matrices and Finite Relations}

\author[E. Akin]{Ethan Akin}
\address{Mathematics Department \\
    The City College \\ 137 Street and Convent Avenue \\
       New York City, NY 10031, USA     }
\email{ethanakin@earthlink.net}

\author[M. Mrozek]{Marian Mrozek}
\address{Division of Computational Mathematics,
  Faculty of Mathematics and Computer Science,
  Jagiellonian University, ul.~St. \L{}ojasiewicza 6, 30-348~Krak\'ow, Poland
}
\email{Marian.Mrozek@uj.edu.pl}

\author[M. Przybylski]{Mateusz Przybylski}
\address{Division of Computational Mathematics,
  Faculty of Mathematics and Computer Science,
  Jagiellonian University, ul.~St. \L{}ojasiewicza 6, 30-348~Krak\'ow, Poland
}
\email{Mateusz.Przybylski@uj.edu.pl}

\author[J. Wiseman]{Jim Wiseman}
\address{Department of Mathematics \\
Agnes Scott College \\ 141 East College Avenue \\ Decatur, GA 30030, USA }
\email{jwiseman@agnesscott.edu}

\begin{abstract}
We give a complete invariant for shift equivalence for Boolean matrices (equivalently finite relations),
in terms of the period, the induced partial order on recurrent components, and the cohomology class of the  relation on those components.
\end{abstract}

\maketitle

\section{Introduction}



Let $X$ be a finite set.  A relation $R$ on $X$ is a subset of $X \times X$ and associated to it is a subshift of finite type,
an important and well-studied dynamical system.  The relation can equivalently be described by a Boolean matrix (a $(0,1)$-matrix for which
matrix addition and multiplication are induced by Boolean algebra, where $1+1=0+1=1$).
For such a relation (or matrix) \emph{shift equivalence} is a natural  equivalence relation with important dynamical significance.
It is strictly weaker than conjugacy, and corresponds roughly to ``eventual conjugacy'' (\cite{LM}).
It  is useful in the classification of such subshifts and it arises in
defining a Conley index for computational approximations of dynamical systems.  Classifications of shift
equivalence are given in \cite{KR1,KR2,MMW}.  Here, we give a complete invariant in terms of the period,
the induced partial order on recurrent components, and the cohomology class of the  relation on those components.

Our main theorem is as follows.  (We define the relevant terms more precisely in
sections \ref{sect:background} and \ref{sect:main}.)  There exists a least integer
$p \ge 1$
such that there exists an integer $N>0$ such that $R^{n+p} = R^n$ for all $n\ge N$; we call $p$ the
\emph{period} of $R$.  Additionally, $R$ induces a partial order $\porder$ on the strongly connected components,
determined by which components map to which.  Finally, a choice of representatives of the strongly connected components
induces a cocycle $\xi: \porder \to \mathcal L_p$, where $\mathcal L_p$ is the
additive monoid of non-empty subsets of $\Z/p\Z$;
we denote the cohomology class of $\xi$ by $[\xi]$.

\begin{thm*}
For a finite relation $R$, the triple $(\porder$, $p$, $[\xi])$ is a complete invariant of shift equivalence.
\end{thm*}

\section{Background and definitions} \label{sect:background}

Let $X$ and $Y$ be sets.  A \emph{relation} 
from $X$ to $Y$
is a subset $R\subseteq X \times Y$.  If $X=Y$ we call $R$ a
\emph{relation on} $X$, and denote the pair by $(X,R)$.  For $x\in X$ and $A\subseteq X$ we define
$$
R(x) := \{y \in Y : (x,y) \in R\},
$$
$$
R(A) := \bigcup\{R(x): x\in A\}.
$$

Given a relation $R\subseteq X \times Y$, the \emph{reverse relation} $R^{-1}\subseteq Y \times X$ is defined by $R^{-1} := \{ (y,x) : (x,y) \in R \}$.

Given another relation $S\subseteq Y\times Z$  we define the \emph{composition} of $S$ with $R$ as the relation
\[
   S\circ R:=\{(x,z)\in X\times Z: \text{ $(x,y)\in R$ and $(y,z)\in S$ for some $y\in Y$}\}.
\]
As with functions, composition of relations is associative.

The identity relation on $X$ is $\id_X=\{(x,x) : x\in X\}$. For $n \in \Z_+$ the \emph{$n$th power} of a relation $R$ in $X$ is given recursively by
\[
   R^n:=\begin{cases}
           \id_X & \text{ for $n=0$,}\\
           R\circ R^{n-1} & \text{ for $n>0$.}\\
        \end{cases}
\]
From associativity it follows that $R^i \circ R^j=R^{i+j}$ for all $i,j \in \Z_+$, the set of nonnegative integers.

Now let $X$ be a finite set.  We can think of a relation on $X$ in various ways:    as a Boolean matrix, a directed graph,
 a multi-valued map, or  a morphism in an appropriately defined category.  We discuss each briefly.

\subsection*{Boolean matrix}
If $X=\{x_1,\dots,x_n\}$, then we can identify a relation $R$  on $X$ with the $n\times n$ Boolean
matrix $A$ with $i,j$ entry given by 1 if $x_i \in R(x_j)$ and 0 otherwise.  It is easy to check
that $A^n$ (where matrix multiplication is Boolean) is the matrix for the relation $R^n$.  Thus
any statement about finite relations has a corresponding linear algebraic statement about Boolean
matrices.
Furthermore, Boolean matrices describe the structure of nonnegative matrices,
since the map sending positive numbers to 1 and 0 to 0 is a homomorphism.  See \cite{K} for details, and \cite{BP} for applications.

\subsection*{Directed graph}

We can identify a finite relation $(X,R)$ with the directed graph
whose vertices are the elements of $X$ and with an edge from $x$ to $y$ if and only
if $y\in R(x)$.  Thus, there is a path in the graph of length $n$ from $x$ to $y$ if and only if $y\in R^n(x)$.
It is often convenient to think of $(X,R)$ in this way and to use graph-theoretic terminology.
The Boolean matrix corresponds to the adjacency matrix for the directed graph.

\subsection*{Multi-valued map}
We can think of a relation as a multi-valued map, where the image of a
point $x$ is the set $R(x)$.  (In general, the image could be empty, but when
considering shift equivalence, we can assume that it contains at least one
point (\cite{MMW}Proposition 13).  This is equivalent to the condition that
every vertex in the directed graph have at least one outgoing edge, or the
condition that every column of the Boolean matrix have at least one 1.)
From a dynamical systems point of view, the powers $R^n$ tell us about the
long-term behavior of the map under iteration.  In this context, finite relations
can arise from ordinary (single-valued) maps in the form of numerical approximations or reconstructions from sampled data.

\subsection*{The category $\RelF$.}
The objects in the category $\RelF$ are finite sets, and the morphisms
from set $X$ to set $Y$ are
the relations from $X$ to $Y$. Composition of
morphisms is composition of relations, and the identity morphism is the identity relation.
This is the viewpoint we will adopt when we consider the Szymczak functor and shift equivalence below.

\subsection*{The category $\End(\RelF)$.}
This is the category of endomorphisms of $\RelF$.
The objects of $\End(\RelF)$ are pairs $(X,R)$, where $X$ is a finite set
(an object in $\RelF$) and $R$ is a relation on $X$ (a morphism from $X$ to $X$).
A morphism $A$ from $(X,R)$ to $(X',R')$ (written $A:(X,R) \to (X',R')$ ) is a relation $A$
from $X$ to $X'$ such that $R'\circ A= A\circ R$.

\begin{defn}
Two finite relations $(X,R)$ and $(X',R')$ are \emph{shift equivalent}
if there exist morphisms $A : (X,R) \to (X',R')$ and $B : (X',R') \to (X,R)$ such that
for some $\ell \in \Z_+$, $B\circ A = R^{\ell}$ and $A\circ B = (R')^{\ell}$.
\end{defn}

If we think of $R$ and $R'$ as Boolean matrices, then this is
just the usual notion of shift equivalence for Boolean matrices (\cite{KR1}).
Working from $\End(\RelF)$ we define a category for which shift equivalence is the isomorphism concept.


\subsection*{The category $\Szym(\RelF)$.}
We can now define the Szymczak category $\Szym(\RelF)$ (cf.~\cite{S}).
The objects of  $\Szym(\RelF)$ are the objects of $\End(\RelF)$.
Given objects $(X,R)$ and $(X',R')$ in $\Szym(\RelF)$, we consider the equivalence relation
given by $(A_1,m_1)\equiv(A_2,m_2)$, for $A_1,A_2:(X,R) \to (X',R')$ and $m_1,m_2
\in \Z_+$, if and only if there exists a
$k\in \Z_+$ such that $A_1\circ R^{m_2+k} = A_2\circ R^{m_1+k}$, or equivalently $(R')^{m_2+k}\circ A_1 = (R')^{m_1+k}\circ A_2$.
The morphisms in  $\Szym(\RelF)$ from $(X,R)$ to $(X',R')$ are the equivalence classes of $\equiv$.
Composition is given by $[A',m']\circ [A,m] = [A'\circ A, m+m']$ and $[\id_X,0]$ is the identity morphism on $(X,R)$.

For example, for any $(X,R) \in \End(\RelF)$ and $m \in \mathbb{N}$, $R^m : (X,R) \to (X,R)$ is a morphism with $[R^m,m] = [id_X,0]$.

The functor $\Szym:\End(\RelF)\to \Szym(\RelF)$ fixes objects and sends a morphism $A$ to $[A,0]$.
\vspace{.5cm}

As shown in \cite{FR}Proposition 8.1,
$(X,R)$ and $(X',R')$ are shift equivalent if and only if they are isomorphic in the Szymczak category. Notice that if
$[A,m] :(X,R) \to (X',R')$ and $[B,n] : (X',R') \to (X,R)$ are inverse isomorphisms in $\Szym(\RelF)$,
then for some $k, \ BAR^k = R^{m+n+k}$ and $AR^kB = AB(R')^k = (R')^{m+n+k}$. So
$A\circ R^k : (X,R) \to (X',R')$ and $B : (X',R') \to (X,R)$ yield the shift equivalence with $\ell = m+n+k$.


Shift equivalence (Szymczak isomorphism) is an important notion dynamically, when
we consider the long-term behavior of a relation.  Very roughly speaking, it
ignores transient behavior and focuses on the dynamics within and between recurrent components.
The same is true when we consider powers of graphs.  In \cite{KR1,KR2}, Kim and Roush consider shift
equivalence of Boolean matrices in part as a step toward disproving the Williams Conjecture from
symbolic dynamics (\cite{KR3,KR4,KR5}).  They provide the following algebraic classification of Boolean matrices up to shift equivalence.

\begin{thm}[\cite{KR2}Proposition 3.5]
Every matrix  is shift equivalent to a periodic power of itself. Two periodic
matrices are shift equivalent if and only if after deletion of dependent rows and columns, they are conjugate by a permutation matrix.
\end{thm}

  Another important application is in defining the discrete Conley index
  for continuous self-maps \cite{S}, which provides information about invariant sets.
  The form of the index relies on choices made in its construction, and it is defined exactly up to shift equivalence.
  In this context it is useful to have a dynamical, rather than algebraic,  classification of finite relations up to shift equivalence.  Such a classification was given in \cite{MMW} as a step toward defining a Conley
  index for such relations, particularly those approximating physical systems.

To state the result from \cite{MMW} on shift equivalence, and to describe the complete invariant in the next section, we need a few more definitions.

For a relation $R$ on $X$, we write $x\to_R y$ if there is a path from $x$ to $y$ in the graph corresponding to $R$, or,
equivalently, if $y\in R^n(x)$ for some
$n \in \nn$, the set of positive integers.  We say that $x$ is \emph{recurrent}
if $x\to_R x$, and denote by $X_R$  the set of recurrent points. We write $x\leftrightarrow_R y$ if $y\to_R x$  and $x\to_R y$.  For the
purposes of shift equivalence, we can assume that every point is recurrent (\cite{MMW}Theorem 11),
so the relation $\leftrightarrow_R$ is an equivalence relation; in accordance with graph theory
usage, we call the equivalence classes the \emph{strongly connected components}.
For a recurrent point $x$, we denote the strongly connected component to which it belongs by $[x]_R$,
and the quotient $X/\leftrightarrow_R$ (the set of strongly connected components) by $\poX$.

Observe that $\to_R$ induces a partial order on $\poX$ as follows:  $[x]_R \ge [y]_R$ if $x\to_R y$, which is
clearly independent of the choice of representatives of $[x]_R$ and $ [y]_R$.  For notational clarity, we
will somewhat redundantly denote the relationship $\ge$ by $R_\ge$, that is, $([x]_R,[y]_R) \in R_\ge$ if and only if $[x]_R \ge [y]_R$.

\begin{defn}
Let $X$ be a finite set.  A relation $R$ on $X$ is in \emph{canonical form} if
\begin{enumerate}
\item every point of $X$ is recurrent,
\item $R$ is a bijection (and hence cyclic) on each strongly connected component, and
\item $R$ is periodic, that is, there exists a least integer $p\ge1$, called the period of $R$, such that $R^{p+1} = R$.
\end{enumerate}
\end{defn}

%

Shift equivalence can be characterized as follows.

\begin{thm}[\cite{MMW}Theorems 12, 13] \label{thm:CanonForm}
Let $(X,R)$ be a finite relation.  Then $(X,R)$ is shift equivalent to a finite relation in canonical form.
Furthermore, two relations $(X,R)$ and $(X',R')$  in canonical form are shift equivalent if and only if
they are isomorphic, that is, there is a bijection $h:X\to X'$ such that $h\circ R = R'\circ h$.
\end{thm}

The reduction to canonical form is constructive. We sketch the steps.

\subsection*{Periodicity}
 Since the infinite sequence $\{R^n \}$ consists of subsets of a finite set, it follows that the set
\begin{equation}\label{eq01}
\Per(R) = \{ q \in \mathbb{N} : R^{q + N(q)} = R^{N(q)} \  \text{for some} \ N(q) \in \mathbb{N} \}
\end{equation}
is nonempty. Clearly, if $n \ge N(q)$, then $R^{q+n} = R^n$.

\begin{defn} An \emph{eventual period} for $R$ is $q \in \mathbb{N}$ such that $R^q \circ R^q = R^q$ and so $R^{q+n} = R^n$ for all $n \ge q$.
In particular, $R^{nq} = R^q$ for all $n \in \mathbb{N}$. \end{defn}

An eventual period $q$ is an element of $\Per(R)$. On the other hand, if $q \in \Per(R)$ and $Nq \ge N(q)$, then $Nq$ is an eventual period.

  \begin{prop} \label{prop:period}
   \begin{itemize}
  \item[(a)] There exists $p \in \nn$ such that $\Per(R) = p\nn$. In particular, if $q$ is an eventual period, then $p|q$.
  We call $p$ the \emph{period} of $R$ (despite the fact that it need not be an eventual period).

  \item[(b)] If $q_1, q_2$ are eventual periods, then $R^{q_1} = R^{q_2}$.

  \item[(c)] If $q$ is an eventual period and $p$ is the period, then $R^{q+p} = R^{q}$.

  \item[(d)] Define $\bar R = R^{q+1}$ where $q$ is an eventual period. For all $n \in \mathbb{N}$,
  \begin{equation}\label{eq01a}
  \bar R^{n} = R^{q+n} = \bar R^{p+n}
 \end{equation}
  In particular, $\bar R^{np} = R^{q}$ for all $n \in \mathbb{N}$.

  \item[(e)] The period $p$ is the minimum element of the set $\{ p' \in \nn : R^{q+p'} = R^q \}$.
  \end{itemize}
  \end{prop}

   \begin{proof} (a) For $q_1, q_2 \in \Per(R),$ if $M \ge N(q_1), N(q_2)$, then $R^{q_1+q_2+M}=R^{q_2+M}=R^M.$ Thus, $\Per(R)$ is an additive sub-semigroup
   of $\mathbb{N}$. Let $p$ be the gcd of $\Per(R)$.  It follows that for sufficiently large $N$, $Np \in \Per(R)$. For a proof of this classical result
   see, e.g. \cite{MMW} Lemma 1. It follows that $p|q$ for all $q \in \Per(R)$. Let $N$ be large enough that $Np$ and $(N+1)p$ are in $\Per(R)$. If
   $M \ge N(Np),N((N+1)p)$, then $R^{p+M} = R^{Np+p+M} = R^{(N+1)p+M} = R^M.$ Thus, $p \in \Per(R)$ and so is every multiple of $p$.

   (b) Assume $q_1 \le q_2$. Since $q_1$ is an eventual period, $R^{q_1+q_2} = R^{q_2}$. Choose $N$ so that $Nq_1 \ge q_2$. Since
   $q_1$ is an eventual period $R^{Nq_1} = R^{q_1}$. Hence, because $q_2$ is an eventual period, $R^{q_1} = R^{Nq_1} = R^{Nq_1+q_2} = R^{q_1+q_2}$.

   (c) Since $p \in \Per(R)$, for large enough $N$, $R^{p+q}=R^{p+Nq}=R^{Nq}=R^q$, where again we use $R^{Nq}=R^q$.

   (d) $\bar R^n = R^{(q+1)n}= R^{nq+n}=R^{q+n}$ since $R^{nq} = R^q$.
    In particular, $\bar R^{p+n}=R^{q+p+n}=R^{q+n}$. Furthermore,
   $\bar R^{np}=R^{q+np}=R^q$ by induction on $n$.

   (e) If $R^{q+p'} = R^q $, then $p' \in \Per(R)$ and so $p|p'$.
   \end{proof}

   We now fix an eventual period $q$.

   It easily follows that $[R^q,q] :(X,R) \to (X,\bar R)$ and $[R^q,q] :(X,\bar R) \to (X,R)$ are inverse isomorphisms in $\Szym(\RelF)$.  For details
   see \cite{MMW} Theorem 9.

   \subsection*{Recurrence}
    Recall that $x$ is recurrent for $R$ if $x \rightarrow_R x$ or, equivalently, $x \in R^n(x)$ for some $n \in \mathbb{N}$.
   It then follows that for any $k \in \mathbb{N}$, $x \in R^{kn}(x)$ and so $x$ is recurrent for $R^k$. Thus, $x$ is recurrent for $R$ if and only if
   it is recurrent for $R^k$. In particular, $x$ is recurrent for $R$ if and only if  for $R^q$ and so if and only if for some $n$,
   $x \in R^{nq}(x) = R^q(x)$. Thus $X_R$,  the set of recurrent points, equals $\{ x : x \in R^q(x) = \bar R^{p}(x) \}$.

   In general, $x \rightarrow_R y$ if $y \in R^n(x)$ for some $n \in \mathbb{N}$. Hence, $x \rightarrow_{\bar R} y$
   if $y \in \bar R^n(x) = R^{q+n}(x)$ for some $n \in \mathbb{N}$. If  $y \in R^n(x)$ and $x \in X_R$, then $x \in R^q(x)$ and so $y \in R^{q+n}(x)$.
   Similarly, if $y \in X_R$, then $y \in R^{q+n}(x)$. Thus, if $x$ or $y$ is recurrent, then $x \rightarrow_R y$ if and only if $x \rightarrow_{\bar R} y$.
   In particular, if $x \leftrightarrow_R y$  then both $x$ and $y$ are recurrent and so $x \leftrightarrow_{\bar R} y$.

    \begin{lemma}\label{lem:recur1} If $x \in X_R$, then $\Per(x) = \{ n \in \mathbb{N}: x \in \bar R^n(x) = R^{q+n}(x) \}$ is nonempty and there
    exists $p'(x) \in \mathbb{N}$ such that $\Per(x) = p'(x)\mathbb{N}$. In particular, $p'(x)|p$. Furthermore, $p$ is the least common
    multiple of the set $\{ p'(x) : x \in X_R \}$.

  In addition, if $x \leftrightarrow_R y$ then $p'(x) = p'(y)$.

    \end{lemma}

     \begin{proof} If $x \in \bar R^{n}(x)$, then $\bar R^m(x) \subset \bar R^{n+m}(x)$.  So if $x \in \bar R^m(x)$ as well, then
      $x \in \bar R^{n+m}(x)$.  That is, $\Per(x)$ is a sub-semigroup of $\mathbb{N}$. Let $p'(x)$ be the gcd. Since $x \in \bar R^p(x)$,
      it follows that $p'(x)|p$. As before, large enough multiples of $p' = p'(x)$ are contained in $\Per(x)$. Hence, for large enough $N$,
      $Np+p' \in \Per(x)$. Hence, $x \in \bar R^{Np+p'}(x) = \bar R^{p'}(x)$.  Thus, $p' \in \Per(x)$ and so $\Per(x) = p' \mathbb{N}$.
      If $\ell$ is the lcm of the $p'(x)$'s, then $\ell$ divides $p$ since each $p'(x)$ does. On the other hand, if
      $y \in R^q(x)$, then $y \in R^{q + np'(x)}(x)$ for all $n \in \Z_+$ and so $y \in R^{q + \ell}(x)$. Conversely, if
      $y \in R^{q + \ell}(x)$, then $y \in R^{q + \ell + np'(x)}(x)$ for all $n \in \Z_+$. We can choose $n$ so that $\ell + np'(x) = q$.
      Thus, $y \in R^{2q}(x) = R^q(x)$. We have shown that $R^{q+\ell} = R^q$. From Proposition \ref{prop:period} (e) it follows that
      $p \le \ell$.

      If $y \leftrightarrow_{\bar R} x$, then for some $i$, $y \in \bar R^{i}(x)$. Since $x \in \bar R^{np'(x)}(x)$, $y \in R^{i+np'(x)}(x)$ for all
      $n \in \mathbb{N}$. On the other hand, for some $j$, $x \in \bar R^j(y)$ and so $y \in R^{i+j+np'(x)}(y)$ for all $n$. Therefore,
      $p'(y)|(i+j+np'(x))$ for all $n$ and so $p'(y)|p'(x)$. Similarly, $p'(x)|p'(y)$.
      \end{proof}

      \begin{lemma}\label{lem:recur2} If $y \in R^{q+i+j}(x)= \bar R^{i+j}(x)$ for some $i,j \in \mathbb{N}$, then there exists $z \in X_R$
      such that $z \in R^{q+i}(x) = \bar R^{i}(x)$ and $y \in R^{q+j}(z) = \bar R^{j}(z)$.\end{lemma}

       \begin{proof} By reducing mod $q$ we may assume $i,j \le q$. Since $R^{Nq+i+j}=R^{q+i+j}$,
       we can choose $N$ so that $(N-2)q$ is greater than the cardinality of $X$ and obtain a
       path $x=x_0, x_1, \dots, x_{Nq+i+j}=y$. It follows that the portion of the path from $x_q$ to $x_{(N-1)q+i+j}$ contains a loop and
       so contains a recurrent point. Insert a cycle of length $q$ beginning and ending at this point. Of course all of the points of
       this cycle are recurrent. Renumbering, we have
       a path $x=x_0, \dots, x_{(N+1)q+i+j}=y$. We can choose a point $z = x_t$ on the cycle so that $t = M_1q + i$ for some $M_1 \ge 1$ and so
       that $((N+1)q+i+j)-t=M_2q +j$ for some $M_2 \ge 1$. Thus, $z \in R^{M_1q+i}(x)=R^{q+i}(x)$ and $y \in R^{M_2q+j}(z)=R^{q+j}(z)$.
         \end{proof}

         Let $\bar R^n|X_R = (\bar R^n) \cap (X_R \times X_R)$.  Using the Lemma and induction on $n$ we obtain
         \begin{equation}\label{eq01b}
           \bar R^n|X_R = (\bar R|X_R)^n
         \end{equation}
           for all $n \in \mathbb{N}$.

           Furthermore, if we define $T = R^q \cap (X \times X_R)$ and $S = R^q \cap (X_R \times X)$, it follows that
           $[T,q] :(X,\bar R) \to (X_R,\bar R|X_R)$ and $[S,q]: (X_R,\bar R|X_R) \to (X,\bar R)$ are inverse isomorphisms in $\Szym(\RelF)$.

           \subsection*{Quotient}
            On the set $X_R$ of recurrent points, we define the equivalence relation
          \begin{equation}\label{eq02}\begin{split}
                     \equiv_R := R^q \cap (R^q)^{-1} = \bar R^p \cap (\bar R^p)^{-1} = \\
                      \{ (x,y) : y \in \bar R^p(x) \ \text{and} \ x \in \bar R^p(y) \}.
                  \end{split}\end{equation}
           Let $\pi : X_R \to X_R/\equiv_R$ be the quotient map to the set of $\equiv_R$ equivalence classes and let
           $\hat R := (\pi \times \pi)(\bar R)$.
           
           \begin{prop} \label{prop:qutient} \begin{enumerate}
            \item[(a)] If $x_1 \equiv_R x_2$ and $y_1 \equiv_R y_2$, then $y_1 \in \bar R^n(x_1)$ if and only if $y_2 \in \bar R^n(x_2)$.
           That is, for all $n \in \mathbb{N}$
           \[
           (\pi \times \pi)^{-1}(\hat R^n) = \equiv_R \circ \bar R^n \circ \equiv_R = \bar R^n.
           \]
            \item[(b)] If $x, y, z$ lie in the same strongly connected component of $X_R$, i.e. each pair is connected by $\leftrightarrow_R$,
           then for any $i \in \mathbb{N}$, if $y , z \in \bar R^i(x)$, then $y \equiv_R z$.\end{enumerate}
           \end{prop}

            \begin{proof} (a) This is obvious from $\bar R^p \circ \bar R^n \circ \bar R^p = \bar R^n$.

            (b) Let $p' = p'(x)$. There exists $j \in \mathbb{N}$ such that $x \in \bar R^j(y)$ and so $x \in \bar R^{i+j}(x)$.
            From the definition of $p'(x)$ we have $p'|(i+j)$.  Write $i+j = ap'$.

            For any $n \in \mathbb{N}$, $x \in \bar R^{np'}(x)$ and so $z \in \bar R^{i+np'}(x)$. It follows that
            $z \in \bar R^{i+j+np'}(y)$.  Since $p'|p$ we can write $p = bp'$. If we choose $n = a(b-1)+b$, then $p|i+j+np'$ and so
            $z \in \bar R^p(y)$.  Similarly, $y \in \bar R^p(z)$ and so $y \equiv_R z$.
             \end{proof}

             It follows from part (a) that $[\pi,0] : (X_R,\bar R|X_R) \to (X_R/\equiv_R,\hat R)$
             and $[\pi^{-1},0] : (X_R/\equiv_R,\hat R) \to (X_R,\bar R|X_R)$
             are inverse isomorphisms in $\Szym(\RelF)$.  Note that distinct $\leftrightarrow_R$ classes are mapped to distinct
             $\leftrightarrow_{\hat R}$ classes.

             From part (b) it follows that $(X_R/\equiv_R,\hat R)$ is canonical.

\section{The main result} \label{sect:main}

Intuitively,  if a relation $R$ on $X$  is in canonical form, with  $p$ as its period and $\porder$ as
the partial order on the strongly connected components, we might like to think of $R$ as a product
$\porder \times \tau_p$ on $\poX \times \Z/p\Z$, where $\tau_p$ is the shift on $\Z/p\Z$, that is,
$\tau_p(t) = t+1 \mod p$.  However, it is more complicated than that, for two reasons.  First,
for a given strongly connected component, the period of the cycle may be a proper divisor of $p$,
and different components may have different periods.  Second, several points of one component
may be related to a single point of another component, or vice versa.  We can address both of these issues as follows.

In the simplest case, when these issues do not arise,  $(X,R)$ is isomorphic to
$\porder \times \tau_p$ on $\poX \times \Z/p\Z$: we can identify a point $x \in X$ with a point
$(a,t) \in  \poX \times \Z/p\Z$, and the image $R(x)$ with the set $\{(b,t+1): (a,b)\in \porder\}$.
  Otherwise, for each point $(a,t) \in \poX \times \Z/p\Z$, we need to keep track of which other
  points are in the image, that is, for each $b\in\porder(a)$, the set of integers $k \mod p$ such that $(b,t+1+k)$ is in the image of $(a,t)$.

  More formally, let $\mathcal L_p$ be the collection of nonempty subsets of $\Z/p\Z$.
  For $F_1$, $F_2 \in \mathcal L_p$, define $F_1+F_2:=\{t_1+t_2:t_1\in F_1, t_2\in F_2\}$.
  Then $\mathcal L_p$ is a commutative monoid with identity $\{0\}$.

\begin{defn}
 A map $\xi:\porder \to \mathcal L_p$ is a \emph{cocycle} for the partial order $\porder$ if it satisfies the condition:
\begin{equation}\label{eq01d} (a,b), (b,c) \in \porder \Rightarrow \xi(a,b) + \xi(b,c) \subseteq \xi(a,c).\end{equation}
 \end{defn}
 (Note that $(a,b), (b,c) \in \porder$ implies that $(a,c)\in \porder$.)  If $\xi$, $\xi'$ are cocycles,
 then $\xi+\xi'$ is a cocycle, and so the cocycles form a monoid with identity $\xi_0$ defined by $\xi_0(a,b) = \{0\}$ for all $(a,b) \in \porder$.

  \begin{lemma}\label{lem:subgroup}
 Let $\xi$ be a cocycle for $\porder$.
  If $(a,b)\in \porder$, then $\xi(a,a)$ and $\xi(b,b)$ are subgroups of $\Z/p\Z$ and
 \begin{equation}\label{eq01e} \xi(a,a) +\xi(a,b) = \xi(a,b) = \xi(a,b) + \xi(b,b). \end{equation}
  \end{lemma}

 \begin{proof}
 For $F\in \mathcal L_p$, we claim that $F+F\subseteq F$ if and only if $F$ is a subgroup of $\Z/p\Z$.
 Clearly, if $F$ is a subgroup, then $0\in F$ and $F+F=F$.  Conversely, if $F+F\subseteq F$, then $F$ is
 closed under addition, and so if $a$ is in $F$, then so are $2a,3a,\ldots$.
 In particular, $0=pa\in F$ and $-a=(p-1)a \in F$.  So $F$ is a subgroup.

 By the definition of cocycle, we have $\xi(a,a) +\xi(a,b) \subset\xi(a,b)$.  Since $0\in\xi(a,a)$, equality holds.
 The proof that $\xi(a,b) = \xi(a,b) + \xi(b,b)$ is similar.
 \end{proof}

 For a cocycle $\xi:\porder \to \mathcal L_p$ and $a \in \poX$, let $q_a$ and $p_a$ be the order and index, respectively, of the
 subgroup $\xi(a,a) \subseteq \Z/p\Z$. Let $\zeta = \bigcap \{ \xi(a,a) : a \in \poX \}$. Let $q_*$ and $p_*$
 be the order and index, respectively, of the  subgroup $\zeta \subseteq \Z/p\Z$.

 \begin{defn} A cocycle $\xi:\porder \to \mathcal L_p$ is called \emph{irreducible} when it satisfies the following equivalent conditions:
 \begin{itemize}
 \item The subgroup $\zeta$ is trivial, i.e. $\bigcap \{ \xi(a,a) : a \in \poX \} = \{ 0 \}$.

 \item The least common multiple of  the set $\{ p_a : a \in \poX \}$ equals $p$.

 \item The greatest common divisor of the set $\{ q_a : a \in \poX \}$ equals $1$.
  \end{itemize}

  Otherwise, $\xi$ is called \emph{reducible}. \end{defn}

 In general we have:
 \begin{align*}
 p_* := \text{the index of}& \ \ \zeta = lcm \{ p_a : a \in \poX \}. \\
 q_* := \text{the order of}& \ \ \zeta = gcd \{ q_a : a \in \poX \}.
 \end{align*}

 \begin{proof} Using the quotient map $\mu_p : \Z \to \Z/p\Z$, $\mu_p^{-1}(\xi(a,a)) = p_a \Z$ for some $p_a \in \nn$. The
 quotient map $\mu_{p_a}$ factors to define a surjection from $\Z/p\Z \to \Z/p_a\Z$ with kernel $\xi(a,a)$. Hence,
 $p_a$ is the index of $\xi(a,a)$ in $\Z/p\Z$. Furthermore, $\mu_p^{-1}(\zeta) = p_* \Z$ where $p_*$ is the lcm of the $p_a$'s.
 Hence, $p_*$ is the index of $\zeta$.

 Now we use the number theory fact that if $p_1,\dots,p_n$ all divide $p$ and $p_*$ is the lcm, then
 $$ p/p_* \ = \ gcd(p/p_1,\dots,p/p_n).$$
 This is trivial for $n=1$ and is an easy exercise for $n=2$. The general result then follows by induction.

 It then follows that the order of $\zeta$, which equals $ \ p/p_*$, is the gcd of the orders $p/p_a$.

 The equivalences in the above definition are now clear.
 \end{proof}

 Observe that if $\xi$ is reducible, then we have a surjection $\rho_{p^*} : \Z/p\Z \to \Z/p_*\Z$ with kernel $\zeta$.
 From Lemma \ref{lem:subgroup} it follows that $\zeta + \xi(a,b) + \zeta = \xi(a,b)$ for all $(a,b)\in \porder$.
 If we define $\xi^*(a,b)$ to be the image of $\xi(a,b)$ under the map $\rho_{p^*}$, then we obtain a
 cocycle $\xi^*:\porder \to \mathcal L_{p^*}$ with $\xi(a,b) = \rho_{p^*}^{-1}(\xi^*(a,b))$ for all $(a,b)\in \porder$.
 We call $\xi^*$ the \emph{irreducible quotient} of $\xi$.

 Intuitively, we think of $\xi(a,b)$ as specifying the points $(b,s)\in \poX\times\Z/p\Z$ that are in the image of $(a,t)$.
 Given a cocycle $\xi$, we define the relation $\Rprodxi$ on $\poX \times \Z/p\Z$ by
 \begin{equation}\label{eq01f}\Rprodxi(a,t) = \{(b,s) : b \in \porder(a), s-t-1\in\xi(a,b)\},\end{equation}
 and more generally  $\Rprodxi^i(a,t) = \{(b,s) : b \in \porder(a), s-t-i\in\xi(a,b)\}$.
 In particular, $\Rprodxi(a,0) = \{(b,s'+1) : b \in \porder(a), s'\in\xi(a,b)\}$.

 \begin{prop} \label{prop:Rproduct}
 $\Rprodxi^i \circ \Rprodxi^j = \Rprodxi^{i+j}$ and $\Rprodxi^{p+i} = \Rprodxi^i$ for $i,j\ge1$.
 Every point of $\poX \times \Z/p\Z$ is recurrent for $\Rprodxi$.
 \end{prop}

 \begin{proof}
 If $(b,s)\in \Rprodxi^j(a,t)$ and $(c,r)\in \Rprodxi^i(b,s)$, then $c \in \porder(a)$ and
 $r-t-(i+j) = (s-t-j)+(r-s-i) \in \xi(a,b) + \xi(b,c) \subseteq\xi(a,c)$.  Thus
 $(c,r) \in \Rprodxi^{i+j}(a,t)$, and thus $\Rprodxi^i \circ \Rprodxi^j \subseteq\Rprodxi^{i+j}$.
 To prove the opposite inclusion, suppose that $(c,r) \in \Rprodxi^{i+j}(a,t)$, so that
 $u=r-t-(i+j) \in \xi(a,c)$.  It is clear that $(c,t+u+j)\in \Rprodxi^j(a,t)$, since
 $t+u+j-t -j = u\in \xi(a,c)$, and $(c,r) \in \Rprodxi^i(c,t+u+j)$, since $r-(t+u+j)-i = 0 \in \xi(c,c)$, by Lemma~\ref{lem:subgroup}.
  Thus $\Rprodxi^{i+j} \subseteq\Rprodxi^i \circ \Rprodxi^j $.

 Because addition is mod $p$ it follows from the definition that $\Rprodxi^{p+i} = \Rprodxi^i$.
 Finally, since $t-t-p = 0 \in \xi(a,a)$, we have $(a,t) \in \Rprodxi^p(a,t)$, and thus every point is recurrent.

 \end{proof}

 We will recover the period of a strongly connected component, which may be a proper divisor of $p$, as follows.
 Let $p_a$ be the index of the subgroup $\xi(a,a)$ in $\Z/p\Z$.  Then the projection from $\Z$ to $\Z/p_a\Z$
 factors to define the surjective homomorphism $\rho_a:\Z/p\Z \to \Z/p_a\Z$ with kernel $\xi(a,a)$.  Thus the period
 of the component is the index $p_a$.

 We can now construct a relation $\Rcanonxi$ in canonical form.  Define

\begin{equation}\label{eq03}\begin{split}
 \Xcanonxi  = \bigcup_{a \in \poX} &\{a\} \times \Z/p_a\Z. \hspace{2cm}
 \\
 \rho: \poX \times \Z/p\Z \to \Xcanonxi &\text{ with $\rho(a,t)=(a,\rho_a(t))$.}
 \\
\Rcanonxi = &\rho\circ\Rprodxi\circ\rho^{-1}. \hspace{2.5cm}
 \end{split}\end{equation}

 \begin{thm}
 If $\xi$ is an irreducible cocycle, then
 $\Rcanonxi$ is a canonical relation  on $\Xcanonxi$ with $\Rcanonxi^{p+i} = \Rcanonxi^i$ for $i\ge1$ and with period $p$.
 \end{thm}

 \begin{proof}
 The strongly connected components are of the form $\{a\} \times \Z/p_a\Z$ for ${a \in \poX}$.  By
 Proposition~\ref{prop:Rproduct} and Lemma~\ref{lem:subgroup}, the restriction of $\Rcanonxi$ to  $\{a\} \times \Z/p_a\Z$ is cyclic of period $p_a$.
 $\Rcanonxi^i = (\rho\circ\Rprodxi\circ\rho^{-1})^i = \rho\Rprodxi^i\rho^{-1}$, again by Lemma~\ref{lem:subgroup},
 and so $\Rcanonxi^{p+i} = \Rcanonxi^i$ by Proposition~\ref{prop:Rproduct}.

 Because the period is $p_a$ on the strongly connected component $\{a\} \times \Z/p_a\Z$, it follows from Lemma \ref{lem:recur1} that the
 period equals the lcm of the $p_a$'s.  This equals $p$ because $\xi$ is irreducible.
 \end{proof}

 Observe that if $\xi$ is reducible,  $\Rcanonxi$ is still a canonical relation  on $\Xcanonxi$  but with period $p_*$ the index of
 the subgroup $\zeta \subseteq \Z/p\Z$.  Furthermore, if $\xi^*:\porder \to \mathcal L_{p^*}$ is the irreducible quotient constructed above,
 it is easy to check that the two cocycles yield same canonical relation  on $\Xcanonxi$.

\begin{ex}
Let $\poX = \{a,b\}$, with $a\ge b$, and let $p=6$.  If $\xi(a,a)=\xi(a,b)=\xi(b,b)=\{0\}$, then
$\Xcanonxi = \poX \times \Z/6\Z$ and $\Rcanonxi = \Rprodxi =\porder \times \tau_6$.  This is the simplest case.
\end{ex}

\begin{ex}
Again, let $\poX = \{a,b\}$, with $a\ge b$, and let $p=6$.  If $\xi'(a,a)= \{0,2,4\}$ and
$\xi'(b,b)=\{0\}$, then $q_a =2$, and $\Xcanon{\xi'} = (\{a\}\times \Z/2\Z) \cup (\{b\}\times \Z/6\Z)$.
If $\xi'(a,b)= \{1,3,5\}$, then $\Rcanon{\xi'}(a,0) = \{(a,1)\} \cup \{b\} \times \{0,2,4\}$ and
$\Rcanon{\xi'}(a,1) = \{(a,0)\} \cup \{b\} \times \{1,3,5\}$.
\end{ex}

\begin{ex} \label{ex:cohomCocycles}
Now let $\poX = \{a,b\}$, with $a\ge b$, and let $p=2$.  If $\xi(a,a)=\xi(a,b)=\xi(b,b)=\{0\}$,
then again we get the simplest case, $\Xcanonxi = \poX \times \Z/2\Z$ and $\Rcanonxi = \Rprodxi =\porder \times \tau_2$.
If instead we have $\xi'(a,b)=\{1\}$ and $\xi'(a,a)=\xi'(b,b)=\{0\}$, then
$\Rcanon{\xi'}(a,0) = \{(a,1)\} \cup \{(b,0)\}$ and $\Rcanon{\xi'}(a,1) = \{(a,0)\} \cup \{(b,1)\}$.
Note that $(\Xcanonxi,\Rcanonxi)$ and $(\Xcanon{\xi'},\Rcanon{\xi'})$ are isomorphic via the bijection that interchanges $(b,0)$ and $(b,1)$.
\end{ex}

 So, given a partial order, a period, and a cocycle, we can construct a relation in canonical form.
 Next, given a relation in canonical form,
 from which we obtain a partial order, period, and cocycle,
 we show that the
 relation constructed as above is
 isomorphic to the original.

Let $(X,R)$ be a relation in canonical form
so that $R$ has a period $p$.  Let $\pi:X\to\poX$
be the projection to the set of strongly connected components; then $\porder = \pi\circ R\circ\pi^{-1}$ is the induced partial order on $\poX$.

Since $R$ is in canonical form, the restriction of $R$ to the strongly connected components is a
bijection; denote this bijection by $f:X \to X$.

\begin{lemma}\label{lem:new} The bijection $f$ is an isomorphism from $(X,R)$ to itself,
i.e. $(x,y) \in R$ if and only if $(f(x),f(y)) \in R$. So for all $i \in \nn$
\begin{equation}\label{eq04}
 f \circ R^i = R^i \circ f = R^{i+1}.
 \end{equation}\end{lemma}

 \begin{proof}
 Observe that $x = f^p(x) \in R^{p-1}(f(x))$ and so if $y \in R(x)$, then $y \in R^p(f(x))$ and $f(y) \in R^{p+1}(f(x)) = R(f(x))$.
 Conversely, if $f(y) \in R(f(x))$, then $y \in R^{p-1}(f(y)) \subset R^p(f(x)) \subset R^{p+1}(x) = R(x)$.

 It immediately follows that $ f \circ R^i = R^i \circ f \subseteq R^{i+1}$.  If $y \in R^{i+1}(x)$, then $f^{p-1}(y) \in R^{p+i}(x) = R^i(x)$
 and so $y = f^p(y) \in f \circ R^i(x)$.

 \end{proof}

 Let $s:\poX \to X$ be a section, that is, a map such that $\pi\circ s =id_{\poX}$.  For $a\in\poX$,
 $s(a)$ is a point in $X$ such that $[s(a)]_R=a$, that is, $a$ is the strongly connected component
 containing $s(a)$, so a section is a choice function for the set $\poX$ of equivalence classes.
For $(a,b)\in\porder$, define
\begin{equation}\label{eq06}
\xi(a,b) = \{ i: f^i(s(b))\in R^{p}(s(a))\}.
\end{equation}

 \begin{thm} \label{thm:CanonIso}
Let $(X,R)$ be a relation in canonical form. If $p$, $\porder$, and $\xi$ are as above, then $\xi$ is an irreducible cocylcle
and $(X,R)$ is isomorphic to $(\Xcanonxi, \Rcanonxi)$.
 \end{thm}

 \begin{proof}
 It follows from Lemma \ref{lem:new} that $\xi$ is a cocycle.

Define the map
$\bar m:\poX \times \Z/p\Z \to X$ by $\bar m(a,t) = f^t(s(a))$, where
$s\colon\poX\to X$ is a section given by $\xi$.  Since $f$ has period $p$, $f^t(s(a))$
depends only on $t \mod p$ and so $f^t$, and hence $\bar m$, are well defined.  We show
that $\bar m \circ \Rprodxi = R \circ \bar m$, that is, $\bar m$ is a morphism from $(\poX \times \Z/p\Z, \Rprodxi)$ to $(X,R)$:
\begin{equation}\label{eq05} \begin{split}
(\bar m \circ \Rprodxi)(a,t) = \bar m(\{(b,r):b \in \porder(a), r-t-1\in\xi(a,b)\})\\
  = \{f^r(s(b)) :b \in \porder(a), r-t-1\in\xi(a,b)\} \hspace{2cm}\\
 = R(f^t(s(a))) = (R \circ \bar m)(a,t),\hspace{3cm}
\end{split}\end{equation}
because $f^{r-t-1}(s(b)) \in R^p(s(a))$ if and only if
$f^r(s(b)) \in R^p(f^{t+1}(s(a)) = R^{p+1}((f^t(s(a))) = R(f^t(s(a)))$ by Lemma \ref{lem:new} again.

By Lemma~\ref{lem:subgroup}, $\bar m$ factors
to define a morphism $m: \Xcanonxi \to X$.
The morphism $m$ is clearly surjective, and it is injective: if $\bar m(a,t) = \bar m(a',t')$,
then $a=a'$ and $f^t(s(a)) = f^{t'}(s(a))$, so that $t-t'\in\xi(a,a)$.  Thus $m$ is an isomorphism $(\Xcanonxi,\Rcanonxi) \to (X,R)$.

Because $p$ is the period of $R$, Lemma \ref{lem:recur1} implies that $p$ the lcm of the orders $p_a$ of the strongly connected components.
It follows that the cocycle is irreducible.
 \end{proof}

 However, the cocycle $\xi$ in Theorem~\ref{thm:CanonIso}  depends on the choice of section.

 \begin{ex} \label{ex:cohomSections}
 We refer back to Example~\ref{ex:cohomCocycles}.  If we choose the section $s$ such that
 $m^{-1}(s(a)) = (a,0)$ and $m^{-1}(s(b))=(b,0)$, then we get the cocycle $\xi$.  However,
 if we choose the section $s'$ such that $m^{-1}(s'(a)) = (a,0)$ and $m^{-1}(s'(b))=(b,1)$, then we get the cocycle $\xi'$.
 \end{ex}

 Note that although the two sections give two different cocycles, the resulting relations are isomorphic.
 Thus we need to find an invariant that does not depend on the choice of section.

 Let $\pi_1:(\Xcanonxi,\Rcanonxi) \to (\poX,\porder)$ be the projection.

 \begin{defn}
 If $\theta:\poX \to \Z/p\Z$ is an arbitrary map, then $\eta(a,b) = \{\theta(b)-\theta(a)\}$
 is a cocycle which we call a \emph{coboundary}.  We say that two cocycles are \emph{cohomologous} if they differ by a coboundary.
 \end{defn}

 Observe that if $\xi$ and $\xi'$ are cohomologous, then the subgroups $\xi(a,a) = \xi'(a,a)$ for all $a \in \poX$. Hence, $\xi$ is irreducible
 if and only if $\xi'$ is. That is, irreducibility is a property of the cohomology class.

\begin{thm} \label{thm:cohom}
For cocycles $\xi$ and $\xi'$, there exists an isomorphism $g:(\Xcanonxi,\Rcanonxi) \to (\Xcanon{\xi'},\Rcanon{\xi'})$
with $\pi_1\circ g = \pi_1$ if and only if $\xi$ and $\xi'$ are cohomologous.
\end{thm}

\begin{proof}
First, assume that they are cohomologous:  $\xi'(a,b) = \xi(a,b) + \{\theta(b)-\theta(a)\}$ for some $\theta$.
Then the map $m_\theta(a,t) := (a, t+\rho_a(\theta(a))$ is an isomorphism
$(\Xcanonxi,\Rcanonxi) \to (\Xcanon{\xi'},\Rcanon{\xi'})$, since $\Xcanonxi$ and
$\Xcanon{\xi'}$ are identical, $m_\theta\circ \Rcanonxi = \Rcanon{\xi'} \circ m_\theta$, and $m_\theta$ has inverse $m_{-\theta}$.

Next, assume that the isomorphism $g$ exists, and choose $\theta(a)$ so that
$g(a,0) = (a, \rho_a(\theta(a)))$.  Then $g(a,\rho_a(t)) = (a,\rho_a(t+\theta(a)))$
and $g(b,\rho_b(t+i))= (b,\rho_b(t+\theta(b)+i))$.
For $(a,b) \in \porder$,
\begin{align*}
((a,\rho_a(t)),(b,\rho_b(t+i+1))) \in  \Rcanonxi \ &\Longleftrightarrow \  i\in \ \xi(a,b) \\
 ((a,\rho_a(s)),(b,\rho_b(t+j+1))) \in \Rcanon{\xi'} \ &\Longleftrightarrow \ j+t-s\in \ \xi'(a,b).
\end{align*}

 Since $(g(a,\rho_a(t)), g(b,\rho_b(t+i+1)))$ is in $\Rcanon{\xi'}$ if and
only if \\ $((a,\rho_a(t)), (b,\rho_b(t+i+1)))$ is in $\Rcanonxi$, it follows that $i+\theta(b) - \theta(a) \in \xi'$.
if and only if $i\in\xi(a,b)$. Thus, $\xi$ and $\xi'$ are cohomologous.
\end{proof}

\begin{ex}
We refer again to Examples~\ref{ex:cohomCocycles} and \ref{ex:cohomSections}.  Observe that
if we set $\theta(a) =0$ and $\theta(b)=1$, then $\xi'(a,b) = \xi(a,b) + \{\theta(b)-\theta(a)\}$, that is, $\xi$ and $\xi'$ are cohomologous.
\end{ex}

We denote the cohomology class of $\xi$ by $[\xi]$.

\begin{thm} \label{thm:CompInv}
For a finite relation in canonical form $(X,R)$, the triple $(\porder, p, [\xi])$ provides a
complete invariant of shift equivalence (and thus of Szymczak isomorphism class), consisting
of the quotient partial order $(\poX,\porder)$, the  period $p>0$
and $[\xi]$, the irreducible cohomology class associated with any section $s:\poX \to X$.
\end{thm}

\begin{proof}
By Theorem~\ref{thm:CanonForm},  two finite relations in canonical form are shift equivalent
if and only if they are isomorphic, and the  triple determines the isomorphism class by
Theorem~\ref{thm:CanonIso}.  To see that the triple is an invariant, let $h:(X,R)\to (X',R')$
be an isomorphism of finite relations in canonical form.  Since $y\in R(x)$ if
and only if $h(y)\in R'(h(x))$, $h$ preserves the partial order $(\poX,\porder)$ and
the  period $p$.  Choose sections $s:\poX \to X$ and $s':\poX \to X'$, which induce
cocycles $\xi$ and $\xi'$ respectively.  By Theorem~\ref{thm:CanonIso}, $(X,R)$ is isomorphic
to $(\Xcanonxi, \Rcanonxi)$ and $(X',R')$ is isomorphic to $(\Xcanon{\xi'}, \Rcanon{\xi'})$,
so $(\Xcanonxi, \Rcanonxi)$ is isomorphic to $(\Xcanon{\xi'}, \Rcanon{\xi'})$ and $[\xi]=[\xi']$ by Theorem~\ref{thm:cohom}.
\end{proof}

\begin{cor}
For any finite relation, the triple $(\porder, p, [\xi])$ provides a complete invariant of
shift equivalence, where the triple is associated to any relation in canonical form shift equivalent to the original.
\end{cor}

\begin{proof}
This follows from  Theorems~\ref{thm:CanonForm} and \ref{thm:CompInv}.
\end{proof}

For a general relation $R$ on $X$ we obtain the invariant via the following steps:
\begin{enumerate}
\item Let  $i,j \in \nn$ such that $R^i = R^{i+j}$ is the first repeat in the sequence $\{R^n\}$. Let $q = nj$, the smallest multiple
of $j$ so that $nj \ge i$. So $q$ is an eventual period. In particular, $R^{q+q} = R^q$.

\item Let $p$ be the smallest element of $\nn$ such that $R^{q+p} = R^q$. So $p$ is the period and $p$ divides $q$.

\item Let $X_R = \{ x \in X : x \in R^q(x) \} = \{ x \in X : x \to_R x \}$. So $X_R$ is the set of recurrent points.

\item Let $\poX$ be the set of $\leftrightarrow_R$ equivalence classes of elements of $X_R$ and $\pi : X_R \to \poX$ be the
projection. So for $x \in X_R$, $\pi(x) = [x]_R$ is the equivalence class containing $x$. On $\poX$ the relation $\to_R$
induces the partial order $\porder$ on $\poX$.

\item Choose a section $\sigma : \poX \to X_R$ for the map $\pi$ so that for $a \in \poX$, $\sigma(a) \in X_R$ with $[\sigma(a)]_R = a$.

\item Define the irreducible cocycle $\xi:\porder \to \mathcal L_p$ by
\begin{equation}
\xi(a,b) = \{ i : [\sigma(b)]_R \cap R^{q+i}(\sigma(b)) \cap R^q(\sigma(a)) \not= \emptyset \}.\end{equation}
That is, $i \in \xi(a,b)$ when there exists $w \leftrightarrow_R \sigma(b)$ with \\ $w \in R^{q+i}(\sigma(b))\cap R^q(\sigma(a))$.

 Observe that if $s(a)$ is the $\equiv_R$ equivalence class of $\sigma(a)$ and $\hat R$ is the relation induced by $R^{q+1}$ on
 the set of $\equiv_R$ classes, then the above condition on $i$ says that $f^i(s(b)) \in \hat R^p(s(a))$. Compare equation (\ref{eq06}).

\item Thus, we have the triple consisting of the period $p$, the partial order $\porder$ on $\poX$, and the irreducible cocycle
$\xi:\porder \to \mathcal L_p$.  This triple is the invariant for $(X,R)$.
 \end{enumerate}%
 
 \subsection*{Acknowledgement}
 
 The research of MM was partially supported by Polish National Science
Center
under Opus Grant 2019/35/B/ST1/00874 and the research cooperation was
funded in part
by the program  Excellence Initiative – Research University at the
Jagiellonian University in Kraków.

The research of MP was partially supported by Polish National Science
Center
under Opus Grant 2019/35/B/ST1/00874.

\bibliographystyle{amsplain}

\end{document}